\documentclass[12pt]{article}
\usepackage{epsfig}
\usepackage{amsfonts}
\usepackage{amssymb}
\usepackage{amsthm}
\usepackage{graphicx}  
\usepackage{url}  
\usepackage[bottom=2.5cm, left=3cm, right=3cm]{geometry}
\newtheorem{theorem}{Theorem}
\newtheorem{lemma}[theorem]{Lemma}

\newcommand{\f}[1]{e^{\frac{2\pi}{\sqrt{3}}\sqrt{#1/\log{#1}}}}

\begin{document}
\title{Counting graphs with different numbers of spanning trees through the counting of prime partitions}
\author{Jernej Azarija \thanks{Department of Mathematics, University of Ljubljana, Jadranska 21, 1000 Ljubljana, Slovenia, e-mail: \texttt{ jernej.azarija@gmail.com}} }
\date{}
\maketitle

\begin{abstract}
Let $A_n$ $(n \geq 1)$ be the set of all integers $x$ such that there exists a connected graph on $n$ vertices with precisely $x$ spanning trees. In this paper, we show that $\left|A_{n}\right|$ grows faster than $\sqrt{n}\f{n}.$ This settles a question of Sedl\'{a}\v{c}ek posed in \cite{SedlNumSp}. 
\end{abstract}

\section{Introduction}
J. Sedl\'{a}\v{c}ek is regarded as one of the pioneers of Czech graph theory. He devoted much of his work to the study of subjects related to the number of spanning trees $\tau(G)$ of a graph $G$. 
In \cite{SedlMinSp} he studied the function $\alpha(n)$ defined as the least number $k$ for which there exists a graph on $k$ vertices having precisely $n$ spanning trees. He showed that for every $n > 6$, it holds:
$$
\alpha(n) \leq \left\{
	\begin{array}{ll}
		\frac{n+6}{3} & \mbox{if } n \equiv 0 \pmod{3},\\
		\frac{n+4}{3} & \mbox{if } n \equiv 2 \pmod{3}.
	\end{array}
\right.
$$

Azarija and \v{S}krekovski \cite{AzRiste} later found out that if $n > 25$ then:

$$
\alpha(n) \leq \left\{
	\begin{array}{ll}
		\frac{n+4}{3} & \mbox{if } n \equiv 2 \pmod{3}, \\
		\frac{n+9}{4} & \mbox{otherwise}.
	\end{array}
\right.
$$

Sedl\'{a}\v{c}ek  continued to study quantities related to the function $\tau$. In \cite{SedlNumSp} and \cite{SedlRegR} he considered the set $B_n^{t}$ defined as the set of integers so that $x \in B_n^{t}$ whenever there is a $t$-regular graph on $n$ vertices with precisely $x$ spanning trees. He showed that for odd integers $t \geq 3$, $ \lim_{a \to\infty} \left |B_{2a}^t\right|= \infty$ and whenever $t \geq 4$ is an even integer $ \lim_{a \to\infty} \left |B_{a}^t\right| = \infty$. In \cite{SedlNumSp} he also studied a more general set - $A_n$ defined as a set of numbers such that $x \in A_n$ whenever there exists a connected graph on $n$ vertices having precisely $x$ spanning trees. One could think about $|A_n|$ as the maximal number of connected graphs on $n$ vertices with mutually different numbers of spanning trees. Using a simple construction, he has shown that:

$$
 \displaystyle \lim_{n\to\infty} \frac{\left|A_n\right|}{n} = \infty
$$
and remarked: {\em it is not clear how the fraction $\frac{\left|A_n\right|}{n^2}$ behaves when $n$ tends to infinity.}
In modern terminology, we could write his result as $\left|A_n\right| = \omega(n)$ since $f(n) = \omega(g(n))$ whenever $|f(n)| \geq c|g(n)|$ for every $c>0$ and $n>n_0$ for an appropriately chosen $n_0$. 

In this paper, we extend his work and show that $\left|A_{n}\right| = \omega(\sqrt{n}\f{n}).$ In order to prove the result we define the graph $C_{x_1,\ldots,x_k}$ as follows. Let $3 \leq x_1 \leq \cdots \leq x_k$ be integers. By $C_{x_1,\ldots,x_k}$ we denote the graph that is obtained after identifying a vertex from the disjoint cycles $C_{x_1}, \ldots, C_{x_k}.$ Since $C_{x_1},\ldots,C_{x_k}$ are the blocks of $C_{x_1,\ldots,x_k}$, it follows that $$\tau(C_{x_1,\ldots,x_k}) = \prod_{i=1}^{k}x_i \quad \hbox{ and } \quad |V(C_{x_1,\ldots,x_k})| = \sum_{i=1}^k x_i - k + 1 .$$
We also introduce some number theoretical concepts. We say that $\left<x_1,\ldots,x_k\right>$ is a {\em partition} of $n$ with integer {\em parts} $1 \leq x_1 \leq \cdots \leq x_k$ if $\sum_{i=1}^{k} x_i = n$. The study of partitions covers an extensive part of the research done in combinatorics and number theory. If we denote by $p(n)$ the number of partitions of $n$ then, the celebrated theorem of Hardy and Ramanujan \cite{HardyRamamujan} states that $$p(n) \sim \frac{1}{4n\sqrt{3}}e^{\pi\sqrt{\frac{2n}{3}}}$$ or, equivalently $$ \displaystyle \lim_{n \to \infty} \frac{p(n)}{ \frac{1}{4n\sqrt{3}}e^{\pi\sqrt{\frac{2n}{3}}}} = 1.$$ Since then, asymptotics for many types of partition functions have been studied \cite{Sedgewick}. For the interest of our paper is the function $p_p(n)$ which we define as the number of partitions of $n$ into prime parts. In \cite{RothSezekres} Roth and Szekeres presented a theorem which can be used to derive the following asymptotic relation for $p_p$ $$p_p(n) \sim \f{n}.$$ Somehow surprising is the fact that if we disallow a constant number of primes as parts, the same asymptotic relation still holds. More specifically, if $p_{op}(n)$ is the number of partitions into odd primes then $$p_{op}(n) \sim p_p(n).$$

\section{Main Result}

In this section we prove that $\left|A_{n}\right| = \omega(\sqrt{n}\f{n})$ by showing that $$\displaystyle \lim_{n \to \infty} \frac{|A_{n}|}{\sqrt{n}\f{n}} = \infty.$$ We do so by estabilishing a lower bound for the number of partitions whose parts are only odd primes and whose sum is less than or equal to a given number $n$. 

\begin{lemma} \label{L1}
Let $P_n$ be the set of all partitions $\left < x_1, \ldots, x_k \right >$ with $\sum_{i=1}^k x_i \leq n$ and all the numbers $x_1, \ldots,x_k$ being odd primes. Then, there exists a $n_0$ such that for all $n \geq n_0$ 
$$ |P_n| \geq \frac{1}{4} \sqrt{n\log{n}}\f{n}. $$
\end{lemma}

\begin{proof}
Let $n_1$ be such a positive integer so that for all $n \geq n_1$ $$p_{op}(n) \geq \frac{1}{2} \f{n}.$$ 
For $n \geq n_1$ we then have $$ |P_n| = \sum_{i=3}^n p_{op}(i) \geq \frac{1}{2}\sum_{i=n_1}^n  \f{i} \geq \frac{1}{2}\int_{n_1-1}^n \f{x} \,dx. $$

Showing that $ \int_{n_1-1}^n \f{x} \,dx \sim \frac{\sqrt{3}}{\pi} \sqrt{n\log{n}} \f{n}$ would imply the existence of a positive integer $n_2$ such that 
$$ \int_{n_1-1}^n \f{x} \,dx \geq \frac{1}{2}\sqrt{n\log{n}} \f{n} $$ for all $n \geq n_2.$ The statement of the theorem would then immediately follow for all $n \geq n_0$ where $n_0 = max\{n_1,n_2\}.$ 
To prove the last asymptotic identity we observe that it follows from 
L'Hospital's rule that $$ \lim_{n \to \infty } \frac{ \int_{n_1-1}^n \f{x} \,dx }{\frac{\sqrt{3}}{\pi}\sqrt{n\log{n}} \f{n}} = \lim_{n \to \infty} \frac {\f{n}}  {\frac{d}{d\,n}(\frac{\sqrt{3}}{\pi}\sqrt{n\log{n}}\f{n} )} = 1.$$ 
\end{proof}

Lemma \ref{L1} readily gives an asymptotic lower bound for $|A_n|.$ 

\begin{theorem} \label{T1}
$|A_n| = \omega(\sqrt{n} \f{n}).$
\end{theorem}

\begin{proof}
Let $P_n$ be defined in the same way as in the statement of Lemma \ref{L1}. To every partition $\left<x_1, \ldots,x_k\right> \in P_n$ with sum $s = \sum_{i=1}^k x_i$ we associate the graph obtained after identifying a vertex from $C_{x_1,\ldots,x_k}$ with a vertex from the disjoint path $P_{n-s+k}.$ Observe that the resulting graph has precisely $n$ vertices and $\prod_{i=1}^k x_i$ spanning trees. Since all the parts in the partitions are primes it follows that any pair of graphs that were obtained from two different partitions in $P_n$ have a different number of spanning trees. Thus $$ |A_n| \geq |P_n| $$ and therefore from Lemma \ref{L1} we know that for large enough $n$ $$ |A_n| \geq \frac{1}{4} \sqrt{n \log{n}} \f{n}. $$ Since $$ \lim_{n \to \infty} \frac{\frac{1}{4} \sqrt{n \log{n}} \f{n}} {\sqrt{n}\f{n}} = \infty$$ it follows from the squeeze theorem that $$ \lim_{n \to \infty}\frac{|A_n|}{\sqrt{n} \f{n}} = \infty $$ from where the stated claim follows.
\end{proof}

Observe, that all the graphs constructed in the proof of Theorem \ref{T1} contain a cut vertex. Since almost all graphs are $2$-connected \cite{Harary} it is reasonable to expect that there exists a construction of a class $C_n$ of $2$-connected graphs of order $n$ with mutually different number of spanning trees such that $|C_n| = \omega(|P_n|) .$ We thus leave it as a further investigation to try to improve the bound on the growth rate of $|A_n|.$

\section{Acknowledgements}

The author is thankful to Riste \v{S}krekovski and Martin Rai\v{c} for fruitful discussions.


\begin{thebibliography}{99}


\bibitem{AzRiste}
J. Azarija, R. \v{S}krekovski, Euler's idoneal numbers and an inequality concerning minimal graphs with a prescribed number of spanning trees, manuscript, 2011.

\bibitem{Sedgewick} P. Flajolet, R. Sedgewick, 
Analytic Combinatorics, 
Cambridge University Press 2009,
p. 574.

\bibitem{HardyRamamujan} 
G. H. Hardy, S. Ramanujan,  
Asymptotic formule in combinatory analysis, 
Proc. London Math. Soc. (2) \textbf{17} (1918) 75--115.

\bibitem{Harary}
F. Harary, E.M Palmer,
Graphical Enumeration
Academic Press, New York, 1973.

\bibitem{RothSezekres}
K. F. Roth,  G. Szekeres, Some asymptotic formulae in the theory of partitions,
Quarterly Journal of Mathematics, Oxford Series \textbf{5} (1954) 241--259.

\bibitem{SedlNumSp} 
J. Sedl\'{a}\v{c}ek, 
On the number of spanning trees of finite graphs, 
\v{C}as. Pro. P\v{est} Mat., Vol. \textbf{94} (1969) 217--221.

\bibitem{SedlMinSp}
J. Sedl\'{a}\v{c}ek, 
On the minimal graph with a given number of spanning trees,
Canad. Math. Bull. \textbf{13} (1970) 515--517.

\bibitem{SedlRegR}
J. Sedl\'{a}\v{c}ek, 
Regular graphs and their spanning trees,
\v{C}as. Pro. P\v{est} Mat., Vol. \textbf{95} (1970) 420--426.

\end{thebibliography}
\end{document}